\documentclass[a4paper,11pt,leqno]{amsart}

\usepackage{amsmath}
\usepackage{amssymb}
\usepackage{hyperref}
     
     \usepackage{mathrsfs}  

\usepackage{graphicx}

\usepackage[usenames,dvipsnames]{pstricks}
\usepackage{epsfig}
\usepackage{pst-grad} 
\usepackage{pst-plot} 

\usepackage{ifthen}
\usepackage{color}

\newcommand{\intav}[1]{\mathchoice {\mathop{\vrule width 6pt height 3 pt depth  -2.5pt
\kern -8pt \intop}\nolimits_{\kern -6pt#1}} {\mathop{\vrule width
5pt height 3  pt depth -2.6pt \kern -6pt \intop}\nolimits_{#1}}
{\mathop{\vrule width 5pt height 3 pt depth -2.6pt \kern -6pt
\intop}\nolimits_{#1}} {\mathop{\vrule width 5pt height 3 pt depth
-2.6pt \kern -6pt \intop}\nolimits_{#1}}}

\def\polhk#1{\setbox0=\hbox{#1}{\ooalign{\hidewidth\lower1.5ex\hbox{`}\hidewidth\crcr\unhbox0}}}

\renewcommand{\div}{\operatorname{div}}

\newcommand{\osc}{\operatorname{osc}}

\renewcommand{\div}{\operatorname{div}}

\newcommand{\spec}{\operatorname{Spec}}
\newcommand{\Span}{\operatorname{Span}}

\newtheorem{theorem}{Theorem}

\newtheorem{definition}{Definition}
\newtheorem{lemma}{Lemma}

\newtheorem{proposition}{Proposition}

\newtheorem{Assumption}{A}

\usepackage{setspace}
\begin{document}

\title{${C}^1$--regularity for  degenerate diffusion equations}
\author{P. Andrade \and D. Pellegrino \and E. Pimentel \and E. Teixeira}

\begin{abstract}

\noindent We prove that any solution of a degenerate elliptic PDE is of class $C^1$ provided the inverse of equation's degeneracy law satisfies an integrability criterium,  viz. $\sigma^{-1} \in L^1\left (\frac{1}{\lambda} {\bf d}\lambda\right )$.  The proof is based upon the construction of a sequence of converging tangent hyperplanes that approximate $u(x)$, near $x_0$, by an error of order $\text{o}(|x-x_0|)$. Explicit control of such hyperplanes are carried over through the construction, yielding universal estimates upon the ${C}^1$--regularity of solutions. Among the main new ingredients required in the proof, we develop an alternative recursive algorithm for renormalization of approximating solutions. This new method is based on a technique tailored to prevent the sequence of degeneracy laws constructed through the process from being, itself, degenerate. 

\vspace{.1in}


\vspace{.05in}


\end{abstract}

\maketitle

\section{Introduction} \label{introduction}
In this article we examine solutions to degenerate nonlinear elliptic equations of the form
\begin{equation}\label{sigma-eq}
	\mathscr{F}( D u, D^2 u)\,=\,f(x) \quad \mbox{in }  B_1.
\end{equation}
The source function $f(x)$ is assumed bounded and the nonlinear operator $\mathscr{F} \colon \mathbb{R}^d \times \mathcal{S}(d) \to \mathbb{R}$ is degenerate elliptic, with law of degeneracy $\sigma$. This means $ \mathscr{F} (\vec{p}, M) = \sigma (|\vec{p}|) F(M),$
for an operator $F\colon \mathcal{S}(d) \to \mathbb{R}$, representing the {\it diffusion agent} of the model and $\sigma$ is a modulus of continuity, otherwise refereed as the {\it law of degeneracy} of the equation; precise definitions will be given later. 

Diffusion processes whose ellipticity is affected by a gradient-dependent term are of fundamental relevance in analysis of partial differential equations. A paramount example --- in the variational setting --- is the $p$-Poisson equation, $\div\left(\mathbf{a}(Du)\right)\,=\,f$ where $\mathbf{a}(\vec{v}) \sim \left |\vec{v} \right |^{p-2} \vec{v}$. The prototype of the theory is the classical $p$-Laplacian, which appears in connection with the optimization problem of the $p$-Dirichlet integral and accounts for a variety of important models in life and social sciences. A robust nonlinear potential theory for treating variational problems with gradient degeneracy has been developed as an offspring of the pioneering work of De Giorgi \cite{DG} and since then it has been a rich and powerful line of research. In this article, though, we are interested in a non-variational counterpart of the theory, to whom Krylov--Safonov work \cite{KS}  plays the role of De Giorgi's. 

Heuristically, the law of degeneracy $\sigma$ impairs the diffusibility attributes of the model near a critical point. The stronger the degeneracy law is, the less efficiently the model diffuses, which in turn affects the smoothing properties of the system. That is, using the natural order for laws of degeneracy:
$$
	\sigma_1  \prec \sigma_2 \quad \text{provided} \quad \sigma_1(t) = \text{o}\left ( \sigma_2(t) \right ),
$$
one should expect that if $\sigma_1  \prec \sigma_2$ then the class of solutions of equations with $\sigma_2$ law of degeneracy should be quantitatively smoother than the corresponding class for $\sigma_1$. A regularity theory for solutions of such equations would then be the mathematical manifestation of diffusibility impairment caused by degeneracy. 

Through such an endeavor, a central question emerges: how much degenerate a diffusion process can be so that differentiability of solutions is yet preserved? 

A classical result obtained independently by Trundiger \cite{Trundiger} and Caffarelli \cite{Caffarelli} asserts that if $\sigma \sim 1$, that is, if the equation is non-degenerate, otherwise termed uniformly elliptic, then solutions are locally of class ${C}^{1,\alpha}$, see \cite{NV, NV2, NV3} for the sharpness of such a regularity. If no condition whatsoever is imposed upon the law of degeneracy $\sigma$, then solutions may fail to be differentiable. In this case the best one can expect is local H\"older continuity, see \cite{Ishii-Lions} and also \cite{D}, \cite{IS2}. The goal of this paper is to examine minimal conditions upon $\sigma$ as to assure solutions retain ${C}^1$--differentiability properties. 

As it happens in many branches of mathematical analysis,  ${C}^1$ estimate is indeed conceptually more difficult as it represents a critical borderline regularity. 

The heuristic discussion above conveys that such a condition should somehow prevent $\sigma(\vec{p})$ from approaching $0$ too abruptly. Our main result captures this insight in a clear and concise format:

\begin{theorem}[Differentiability of solutions]\label{main_theorem1}
Let $u \in {C}^{0}(B_1)$ be a viscosity solution of a diffusive process
\begin{equation}\label{eq_fnldeg}
\mathscr{F}( D u, D^2u) = f \quad \mbox{in} \quad B_1,
\end{equation}
where $\mathscr{F}(\vec{p},M) = \sigma(|\vec{p}|) F(M)$, where $\sigma(\cdot)$ is a modulus of continuity with inverse $\sigma^{-1}$. Assume the diffusion agent $F$ is uniformly elliptic and the law of degeneracy $\sigma$ is such that  $\sigma^{-1} \in L^1\left ((0, \tau);\frac{1}{\lambda}  {\bf d} \lambda \right )$. Then $u \in {C}^1_{loc}(B_1)$ and there exists a modulus of continuity $\omega \colon \mathbb{R}^+_0\to\mathbb{R}^+_0$ depending only upon  dimension, ellipticity constants, $\sigma$, $\left\|u\right\|_{L^\infty(B_1)}$, and $\left\|f\right\|_{L^\infty(B_1)}$ such that
\[
	 \left|Du(x)\,-\,Du(y) \right|\,\leq\,\omega\left(|x-y|\right),
\]
for every $x, y \in B_{1/4}$.  
\end{theorem}

In particular, if $\sigma^{-1}$ behaves as a H\"older continuous function near the origin, then solutions are in fact locally ${C}^{1,\gamma}$, for some $0< \gamma < 1$. This accounts for non-linear elliptic PDEs with power--like degeneracy laws, $\sigma(|\vec{p}|) = \text{O} \left ( |\vec{p}|^M \right )$, as $\vec{p} \to 0$, for some $M>0$, and thus Theorem \ref{main_theorem1} extends the results from \cite{Imbert-Silvestre2013}, see also \cite{Araujo-Ricarte-Teixeira2015}. 

We conclude the introduction by describing the essences of the strategy  for  proving Theorem \ref{main_theorem1}. Given a point $x_0 \in B_{1/4}$, we want to attain the existence of a tangent hyperplane $\mathscr{H}_{x_0} = \ell_{x_0}^{-1}(0)$ and a modulus of continuity $\omega$ such that
$$
	\sup\limits_{x\in B_\gamma (x_0)} \left |u(x) - \ell_{x_0}(x)  \right | \le \gamma \omega(\gamma), 
$$
for all $0< \gamma \le 1/4$. This is achieved by means of a geometric recursive construction. Given a family of laws of degeneracy $\Sigma$, define the functional space $\Xi_{\epsilon, \lambda, \Lambda, \Sigma}$ to be the set of all continuous functions $ u\in C(B_1) $ such that  $\|u\|_\infty \le 1$   and  
$$
	\left |  \mathscr{F}( D u, D^2u) \right | < \epsilon
$$
 in the viscosity sense, for an operator $\mathscr{F}(\vec{p},M) = \sigma(\vec{p}) F(M)$, with $F$ $(\lambda, \Lambda)$ and $\sigma \in \Sigma$. Then for some positive $\beta>0$ there exists a modulus of continuity $\tau$ such that
\begin{equation}\label{idea1}
 	 \Xi_{\epsilon, \lambda, \Lambda, \Sigma} \Big |_{B_{1/2} }  \subset \mathscr{N}_{\tau(\epsilon)} \left ( {C}^{1,\beta}(B_{1/2}) \right ) 
\end{equation}
where $ \Xi_{\epsilon, \lambda, \Lambda, \Sigma} \Big |_{B_{1/2} }$ simply represents the restriction of functions in $ \Xi_{\epsilon, \lambda, \Lambda, \Sigma}$ to $B_{1/2} $ and  $\mathscr{N}_{\tau(\epsilon)} \left ( {C}^{1,\beta}(B_{1/2}) \right ) $ is the $\tau(\epsilon)$---neighborhood of ${C}^{1,\beta}(B_{1/2}) $ within $L^\infty(B_{1/2})$. 

Here comes the first main technical difficulty of the proof. To attain such a pivotal result, one must require a sort of ``non-collapsing" property upon the family of laws of degeneracy. Otherwise, if one does not prevent a sequence of  laws of degeneracy $\sigma_j$ to converge to a function $\sigma_\infty$ which vanishes identically on a non-trivial interval $[0,\delta]$, $\delta > 0$, then any function whose Lipschitz norm is less than $\delta$ would belong to the limit set of solutions and \eqref{idea1} couldn't hold true. The concept of non-collapsing moduli of continuity is introduced in Section \ref{sct non-collapsing} and the approximation scheme is the content of Proposition \ref{approx_lemma}.

Once such a result is available, the idea is to iterate it, using supporting hyperplanes of ${C}^{1,\beta}(B_{1/2})$ functions that are close enough to a scaled version of the preceding element of the sequence. To put forward such strategy, though, one has to tackle two intrinsic difficulties. The first one is that $u$ subtracted an affine function solves a family of equations parametrized by a non-compact set of parameters, for which one nonetheless has to extract some compactness property. This is attained by classical PDE methods, inherent of the viscosity theory, see Proposition \ref{teo_lipschitz}. The second, and most challenging difficulty is that these corresponding PDEs are now ruled by a new family of degeneracy laws, which could be collapsing. The main novelty here is a new algorithm for choosing the normalization in each step,  based on a sort of ``shoring-up" technique, which effectively prevents the resulting degeneracy laws from collapsing.

\section{Some preliminaries}

We start off this section by detailing the main assumptions and setting up  some notions and results that will be used in the paper.

For a non-linear operator 
\begin{equation}\label{structure}
	\mathscr{F}(\vec{p}, M) = \sigma \left ( \left |\vec{p} \right | \right )F(M),
\end{equation}
we call $\sigma$ its {\it degeneracy law} and $F$ its {\it diffusion agent}. This latter nomenclature is justified by the ellipticity condition of $F$. This is the content of our first main assumption:

\begin{Assumption}\label{assump_F}\rm
We suppose  $F \colon \mathcal{S}(d)\to\mathbb{R}$ is a $(\lambda,\Lambda)$-uniformly elliptic operator. That is, there exist $0<\lambda\leq\Lambda$ such that
\begin{equation}\label{eq_ellipticity}
	\lambda\left\|N\right\|\,\leq\,F(M\,+\,N)\,-\,F(M)\,\leq\,\Lambda\left\|N\right\|,
\end{equation}
for every $M,\,N\in\mathcal{S}(d)$ with $N\geq 0$. Further, we suppose $F(0)=0$.
\end{Assumption}

As usual, $\mathcal{S}(d)$ stands for the space of $d\times d$ symmetric matrices endowed with its natural order. 

Next we make precise the condition required on the degeneracy law of the operator in \eqref{structure}, which is key for our differentiability result.

\begin{Assumption}\label{assump_dini}
We suppose $\sigma \colon [0,+\infty)\to[0,+\infty)$ is a modulus of continuity for which its inverse, $\sigma^{-1} \colon \sigma\left([0,+\infty)\right)\to[0,+\infty)$ satisfies the Dini condition \eqref{DC}.
\end{Assumption}

Throughout the whole paper we shall also assume
\begin{equation}\label{normalization-sigma}
	\sigma(1) \ge 1,
\end{equation}
which is a mere normalization. 

We recall by modulus of continuity we simply mean an increasing function $\varphi(t)$ defined over an interval of the form $(0, T]$, or over the whole $\mathbb{R}_0^{+} :=(0,+\infty)$, into $\mathbb{R}_0^{+}$ such that  $\lim\limits_{t\to 0} \varphi(t) = 0.$

Given its importance to our main theorem, below we introduce the formal definition of Dini condition:

\begin{definition} \label{def_dini}
A modulus of continuity $\omega$  is said to satisfy the Dini condition if 
\begin{equation}\label{DC}
	\int_0^\tau\frac{\omega (t)}{t}{\bf d}t\,<\,+\infty,
\end{equation}
for some $\tau>0$. 
\end{definition}

The Dini condition plays an important role in mathematical analysis, notably in harmonic analysis and its applications to the theory of PDEs. Recall a function $f\colon X \to Y$ defined over a metric space $(X, d_X)$ into another metric space $(Y, d_Y)$ is said to be Dini continuous if:
$$
	d_Y \left ( f(x_1), f(x_2) \right ) \le \omega_f  \left ( d_X(x_1, x_2) \right ),
$$
for a modulus of continuity  $\omega_f$ satisfying  the Dini condition \eqref{DC}. {For the sake of precision, it is convenient to define the modulus of continuity of $f$ as%
\[
\omega_{f}(t)=\sup_{d_{X}(x_{1},x_{2})\leq t}d_{Y}\left(  f(x_{1}),f(x_{2})\right).
\]}
Obviously any H\"older continuous function $h$ is Dini continuous, as its modulus of continuity is given by $\omega_h(d) = Cd^\alpha$ and
\[
	\int_0^1\frac{\omega_h(t)}{t}{\bf d}t\,=\,C\alpha^{-1},
\]
which is finite. There are however many important examples of  Dini continuous functions that are not H\"older continuous. A classical family of examples is given by:
\begin{equation}\label{Example1}
	\phi_\alpha (t) =  \left ( \frac{1}{1-\ln t} \right )^\alpha,
\end{equation} 
for $  \alpha>1$. Further examples of Dini continuous functions can be crafted through generalized power series.  Let $\left (\gamma_k \right )_{k\in \mathbb{N}} \in c_0$ be a sequence of positive numbers converging to zero and $\left (a_k \right )_{k\in \mathbb{N}} \in \ell_1$ be sequence of positive numbers. Define
$$
	\omega(t) = \sum\limits_{j=1}^\infty a_jt^{\gamma_j}.
$$
Assume for some $  t_\star>0$ the series is convergent  at $t=t_\star$ and that, 
$$
	 \sum\limits_{j=1}^\infty a_j \frac{\tau^{\gamma_j}}{\gamma_j} < \infty,
$$
for some $0<\tau <t_\star$. Then $\omega(t)$, defined over $(0, t_\star]$, verifies the Dini condition. For instance, $\omega(t) =  \sum\limits_{j=1}^\infty \frac{\sqrt[j]{t}}{2^j}$  satisfies the Dini condition. Notice that all examples built up through this method fail to be $\epsilon$--H\"older continuous for all $0< \epsilon<1$.

Similarly, there are a plethora of Dini moduli of continuity, $\tilde{\phi}$ verifying $\phi_{\alpha}(t)  = \text{o}(\tilde{\phi}(t))$ for all $\alpha > 1$, where $\phi_{\alpha}$ are the standard examples from \eqref{Example1}. For instance,
$$
	\tilde{\phi}(t) := \sum\limits_{n=1}^\infty a_n \left ( \frac{1}{1-\ln t} \right )^{1+\frac{1}{n}}, 	
$$
where $a_n = \frac{1}{2^nb_n}$, for $b_n := \int_0^{1 } \frac{1}{t} \left ( \frac{1}{1-\ln t} \right )^{1+\frac{1}{n}} {\bf d}t <+\infty$.

\medskip

Dini condition can also be characterized in terms of the summability of $\omega$ evaluated along geometric sequences. That is, a modulus of continuity $\omega$ satisfies the Dini condition \eqref{DC} if, and only if, 
\begin{equation}\label{eq_dinicharact}
	\sum_{n=1}^\infty\omega (\tau \cdot \theta^n)\,<\,\infty,
\end{equation}
for every $\theta\in(0,1)$. Indeed, by elementar partition argument, there exist points $\xi_i \in [\tau  \theta^{i}, \tau \theta^{i-1}]$ such that:
\begin{equation}\label{partition}
	 \left ( 1-\theta \right )  \sum\limits_{i=1}^\infty   \omega (\xi_i) \le \int_0^\tau\frac{\omega (t)}{t}{\bf d}t \le \frac{1-\theta}{\theta}  \sum\limits_{i=1}^\infty   \omega (\xi_i).
\end{equation}
We resort to the characterization in \eqref{eq_dinicharact} further in our arguments.

Finally, we present our assumption concerning the source term $f$. 

\begin{Assumption}\label{assump_f}\rm
We suppose $f\in L^\infty(B_1)$. 
\end{Assumption}

Our main Theorem \ref{main_theorem1} asserts that if $u$ satisfies an equation 
$$
	\mathscr{F}(D u, D^2u) = f \quad \mbox{in} \quad B_1,
$$
in the viscosity sense, where $\mathscr{F}(\vec{p},M) = \sigma(|\vec{p}|) F(M)$ and assumptions A1--A3 are in order, then $u \in {C}^1_{\text{loc}}(B_1)$ with universal estimates.

Next, on account of completeness, we proceed by recalling the notion of viscosity solution used in the paper. 

\begin{definition}\label{def_viscsol}
Let $G \colon B_1\times\mathbb{R}^d\times\mathcal{S}(d)\to\mathbb{R}$ be a degenerate elliptic operator. We say $u\in{C}(B_1)$ is a viscosity sub-solution to 
\begin{equation}\label{eq_defvisc}
	G(x,Du,D^2u)\,=\,0\hspace{.4in}\mbox{in}\hspace{.1in}B_1
\end{equation}
if for every $x_0\in B_1$ and $\varphi\in{C}^2(B_1)$ such that $u-\varphi$ attains a local maximum in $x_0$, we have
\[
	G(x_0,D\varphi(x_0),D^2\varphi(x_0))\,\leq\,0.
\]
Conversely, we say $u\in{C}(B_1)$ is a viscosity super-solution to \eqref{eq_defvisc} if for every $x_0\in B_1$ and $\varphi\in{C}^2(B_1)$ such that $u-\varphi$ attains a local minimum in $x_0$, we have
\[
	G(x_0,D\varphi(x_0),D^2\varphi(x_0))\,\geq\,0.
\]
In case $u$ is a viscosity sub- and super-solution to \eqref{eq_defvisc}, we say $u$ is a viscosity solution to the equation.
\end{definition}

Throughout the paper, we say that $u\in L^\infty(B_1)$ is a normalized solution if $\|u\|_{L^\infty(B_1)}\leq 1$. For a comprehensive account of the theory of viscosity solutions, we refer the reader to \cite{Crandall-Ishii-Lions1992} and \cite{Caffarelli-Cabre1995}. 
\section{$c_0$--modulators in $\ell_1$} 

In this section we establish the first key technical lemma needed in the proof of Theorem \ref{main_theorem1}. It concerns the existence of $c_0$ sequences, $\left ( c_j \right )_{j\in \mathbb{N}}$, whose rate of convergence to zero is carefully modulated in such a way $\frac{a_j}{c_j} \sim  {a_j}$ in $\ell_1$.  

We will use this lemma as to sponsor the existence of ``shored-up" laws of degeneracy, see Section \ref{sct non-collapsing},  for which the tangential analysis from Section \ref{sec_gta} can be employed, while not destroying the convergence of approximating hyperplanes. Here is its precise statement.

\begin{lemma}  \label{DP} Given any sequence of summable numbers $\left(  a_{j}\right)  _{j\in
\mathbb{N}}\in\ell_{1}$ and $\epsilon,\delta>0,$ there is a sequence $\left(
c_{j}\right)  _{j\in\mathbb{N}}\in c_{0}$, satisfying
\[
\max_{j\in\mathbb{N}}\left\vert c_{j}\right\vert \leq\epsilon^{-1}%
\]
such that%
\[
\left(  b_{j}\right)  _{j\in\mathbb{N}}:=\left(  \frac{a_{j}}{c_{j}}\right)
_{j\in\mathbb{N}}\in\ell_{1}%
\]
and
\[
\epsilon\left(  1-\frac{\delta}{2}\right)  \left\Vert \left(  a_{j}\right)
\right\Vert _{\ell_{1}}\leq\left\Vert \left(  b_{j}\right)  \right\Vert
_{\ell_{1}}\leq\epsilon\left(  1+\delta\right)  \Vert \left(  a_{j}\right) \Vert_{\ell_{1}}.
\]
\end{lemma}

\begin{proof}
Let $\delta>0$. Starting off with the hypothesis
\[
\left(  a_{j}\right)  _{j\in\mathbb{N}}\in\ell_{1},
\]
let $n_{1}$ be an integer such that%
\[%
{\textstyle\sum\limits_{k=n_{1}}^{\infty}}
\left\vert a_{k}\right\vert <\frac{\delta\left\Vert \left(  a_{j}\right)
\right\Vert _{\ell_{1}}}{2}.
\]
In what follows, let $n_{2}>n_{1}$ be such that%
\[%
{\textstyle\sum\limits_{k=n_{2}}^{\infty}}
\left\vert a_{k}\right\vert <\frac{\delta\left\Vert \left(  a_{j}\right)
\right\Vert _{\ell_{1}}}{2^{3}};
\]
and, in general, let $n_{j}>n_{j-1}$ be such that%
\[%
{\textstyle\sum\limits_{k=n_{j}}^{\infty}}
\left\vert a_{k}\right\vert <\frac{\delta\left\Vert \left(  a_{j}\right)
\right\Vert _{\ell_{1}}}{2^{2j-1}}%
\]
for all $j.$ Next we construct the sequence of positive numbers $c_{j}$ as
follows:
\begin{align*}
c_{1} &  =\cdots=c_{n_{2}-1}=\frac{1}{\epsilon},\\
c_{n_{2}} &  =\cdots=c_{n_{3}-1}=\frac{1}{2\epsilon},\\
c_{n_{3}} &  =\cdots=c_{n_{4}-1}=\frac{1}{2^{2}\epsilon},\\
&  \quad \quad  \quad  \vdots\\
c_{n_{j}} &  =\cdots=c_{n_{j+1}-1}=\frac{1}{2^{j-1}\epsilon}%
\end{align*}
and so on. Thus, by the very construction, $\left(c_{j}\right)
_{j\in\mathbb{N}}\in c_{0}$ and%
\[
\max_{j\in\mathbb{N}}\left\vert c_{j}\right\vert \leq\epsilon^{-1}.
\]
Next we estimate, for all $j\geq1$:%
\begin{align*}%
{\textstyle\sum\limits_{k=n_{j}}^{n_{j+1}-1}}
\left\vert \frac{a_{k}}{c_{k}}\right\vert  &  =%
{\textstyle\sum\limits_{k=n_{j}}^{n_{j+1}-1}}
\left\vert \frac{a_{k}}{1/2^{j-1}\epsilon}\right\vert \\
&  <2^{j-1}\epsilon%
{\textstyle\sum\limits_{k=n_{j}}^{n_{j+1}-1}}
\left\vert a_{k}\right\vert \\
&  <2^{j-1}\epsilon\frac{\delta\left\Vert \left(  a_{j}\right)  \right\Vert
_{\ell_{1}}}{2^{2j-1}}\\
&  =\frac{\epsilon\delta\left\Vert \left(  a_{j}\right)  \right\Vert
_{\ell_{1}}}{2^{j}}.
\end{align*}
Hence%
\begin{align*}%
{\textstyle\sum\limits_{k=1}^{\infty}}
\left\vert \frac{a_{k}}{c_{k}}\right\vert  &  =%
{\textstyle\sum\limits_{k=1}^{n_{1}-1}}
\left\vert \frac{a_{k}}{c_{k}}\right\vert +%
{\textstyle\sum\limits_{k=n_{1}}^{\infty}}
\left\vert \frac{a_{k}}{c_{k}}\right\vert \\
&  < \epsilon\left\Vert \left(a_{j}\right) \right\Vert_{\ell_{1}}+\epsilon\delta\left\Vert \left(
a_{j}\right)  \right\Vert _{\ell_{1}}\\
&  =\epsilon\left(  1+\delta\right)  \left\Vert \left(a_{j}\right)\right\Vert_{\ell_{1}}%
\end{align*}
On the other hand, since%
\[
{\textstyle\sum\limits_{k=1}^{n_{1}-1}}
\left\vert a_{k}\right\vert +%
{\textstyle\sum\limits_{k=n_{1}}^{\infty}}
\left\vert a_{k}\right\vert =\left\Vert \left(a_{j}\right)\right\Vert_{\ell_{1}}%
\]
and%
\[%
{\textstyle\sum\limits_{k=n_{1}}^{\infty}}
\left\vert a_{k}\right\vert <\frac{\delta\left\Vert \left(  a_{j}\right)
\right\Vert _{\ell_{1}}}{2}%
\]
we have%
\[%
{\textstyle\sum\limits_{k=1}^{n_{1}-1}}
\left\vert a_{k}\right\vert >\left\Vert \left(  a_{j}\right)  \right\Vert
_{\ell_{1}}-\frac{\delta\left\Vert \left(  a_{j}\right)  \right\Vert
_{\ell_{1}}}{2}%
\]
Therefore, we obtain%
\begin{align*}%
{\textstyle\sum\limits_{k=1}^{\infty}}
\left\vert \frac{a_{k}}{c_{k}}\right\vert  &  \geq%
{\textstyle\sum\limits_{k=1}^{n_{1}-1}}
\left\vert \frac{a_{k}}{c_{k}}\right\vert \\
&  =\epsilon%
{\textstyle\sum\limits_{k=1}^{n_{1}-1}}
\left\vert a_{k}\right\vert \\
&  >\epsilon\left(  \left\Vert \left(  a_{j}\right)  \right\Vert _{\ell_{1}%
}-\frac{\delta\left\Vert \left(  a_{j}\right)  \right\Vert _{\ell_{1}}}%
{2}\right)  \\
&  =\epsilon\left\Vert \left(  a_{j}\right)  \right\Vert _{\ell_{1}}\left(
1-\frac{\delta}{2}\right),
\end{align*}
and the lemma is finally proven. 
\end{proof}

Note, in general, one is not allowed to let $\delta\rightarrow0$ as the sequence $\left(  c_{j}\right)_{j\in\mathbb{N}}$ depends itself upon $\delta$.  Also, it is not hard to verify that the sequence $\left(c_{j}\right)_{j\in\mathbb{N}}$ in Lemma \ref{DP} cannot be taken universal with respect to the $\ell_1$ norm of $\left(a_{j}\right)_{j\in\mathbb{N}}$. This is because for all $\left(  c_{j}\right)  _{j\in\mathbb{N}}\in c_{0}$ it is possible to find a monotone sequence 
$\left(  a_{j}\right)  _{j\in\mathbb{N}}\in\ell_{1}$ with, say $\left\Vert \left(
a_{j}\right)  _{j\in\mathbb{N}}\right\Vert_{\ell_{1}}=1$, such that%
\[%
{\textstyle\sum\limits_{k=1}^{\infty}}
\left\vert \frac{a_{k}}{c_{k}}\right\vert =\infty.
\]
Indeed, for $k \in \mathbb{N}$ let $n_{k}$ be such that%
\[
\left\vert c_{j}\right\vert <\frac{1}{2^{2k+3}}%
\]
for all $j\geq n_{k}.$ Note that we can always choose $n_{1}<n_{2}%
<n_{3}<\cdots$ satisfying%
\[
n_{1}-1<n_{2}-n_{1}<n_{3}-n_{2}<n_{4}-n_{3}<\cdots.
\]
Next define: 
$$
	a_{n_{k}}  = a_{n_{k}+1} = a_{n_{k}+2} = \cdots = a_{n_{k+1}-1} = \frac{1}{2^{k+1}\left(  n_{k+1}-n_{k}\right)  }.
$$
Finally, taking $ a_{1}=\cdots=a_{n_{1}-1}=\frac{1}{2\left(  n_{1}-1\right)  }$, 
we have $a_{j+1}\leq a_{j}$ for all $j$ and $ {\displaystyle\sum\limits_{j=1}^{\infty}} \left\vert a_{j}\right\vert =1.$
Furthermore,%
\[%
{\textstyle\sum\limits_{k=1}^{\infty}}
\left\vert \frac{a_{k}}{c_{k}}\right\vert \geq%
{\textstyle\sum\limits_{k=1}^{\infty}}
{\textstyle\sum\limits_{j=n_{k}}^{n_{k+1}-1}}
\left\vert \frac{a_{j}}{c_{j}}\right\vert \geq%
{\textstyle\sum\limits_{k=1}^{\infty}}
\frac{\left(  2^{k+1}\left(  n_{k+1}-n_{k}\right)  \right)  ^{-1}}{\left(
2^{2k+3}\right)  ^{-1}}\left(  n_{k+1}-n_{k}\right)  =\infty.
\]

\medskip

Such an impossibility has implications insofar as the constants dependence on Theorem \ref{main_theorem1} is concerned. In particular one should not expect to control the $C^1$--regularity of solutions solely by $\|\sigma^{-1}\|_{L^1(\lambda^{-1} {\bf d}\lambda)}$, see last Section.

In the next Lemma we tackle the problem of obtaining a universal sequence $\left(c_{j}\right)_{j\in\mathbb{N}}$ for a compact set of $\ell_1$. The proof is an adaptation of the one put forward in Lemma \ref{DP}, so we just sketch it here.

\begin{lemma}\label{DP-Comp}
Let $A$ be a compact subset of $\ell_{1}$ with $0\notin A.$ Given $\epsilon,\delta>0,$ there is a
universal sequence $\left(  c_{j}\right)  _{j\in\mathbb{N}}\in c_{0}$ with
\[
\max_{j\in\mathbb{N}}\left\vert c_{j}\right\vert \leq\epsilon^{-1}%
\]
such that%
\[
\left(  b_{j}\right)  _{j\in\mathbb{N}}:=\left(  \frac{a_{j}}{c_{j}}\right)
_{j\in\mathbb{N}}\in\ell_{1}%
\]
and
\[
\epsilon\left(  1-\frac{\delta}{2}\right)  \left\Vert \left(  a_{j}\right)
\right\Vert _{\ell_{1}}\leq\left\Vert \left(  b_{j}\right)  \right\Vert
_{\ell_{1}}\leq\epsilon\left(  1+\delta\right)  \Vert \left(  a_{j}\right) \Vert_{\ell_{1}},
\]
for all $\left(a_{j}\right)  _{j\in\mathbb{N}}\in A$.
\end{lemma}

\begin{proof}  We start off the proof by noticing that if $A\subset\ell_{1}$ is compact then,
for all $\varepsilon>0$, there is an $n_{\varepsilon}$ such that%
\[%
{\textstyle\sum\limits_{j=n_{\varepsilon}}^{\infty}}
\left\vert a_{j}\right\vert <\varepsilon,
\]
for all $\left(a_{j}\right)  _{j\in\mathbb{N}}\in A$.
%
%
Next, let $\delta>0$ and set%
\[
r=\inf_{x\in A}\left\Vert x\right\Vert >0.
\]
Let $n_{1}$ be an integer such that%
\[%
{\textstyle\sum\limits_{k=n_{1}}^{\infty}}
\left\vert a_{k}\right\vert <\frac{\delta r}{2}%
\]
for all $\left(  a_{k}\right)  \in A.$

In what follows, for $j\in\mathbb{N}$, let $n_{j}>n_{j-1}$ be such that%
\[%
{\textstyle\sum\limits_{k=n_{j}}^{\infty}}
\left\vert a_{k}\right\vert <\frac{\delta r}{2^{2j-1}}%
\]
for all $\left(  a_{k}\right)  \in A$ and for all $j.$ Next we construct the
sequence of positive numbers $c_{j}$ as follows:
\begin{align*}
c_{1} &  =\cdots=c_{n_{2}-1}=\frac{1}{\epsilon},\\
c_{n_{2}} &  =\cdots=c_{n_{3}-1}=\frac{1}{2\epsilon},\\
c_{n_{3}} &  =\cdots=c_{n_{4}-1}=\frac{1}{2^{2}\epsilon},\\
& \quad \quad \quad  \vdots\\
c_{n_{j}} &  =\cdots=c_{n_{j+1}-1}=\frac{1}{2^{j-1}\epsilon}%
\end{align*}
The proof now follows the reasoning of Lemma \ref{DP}.
\end{proof}

\section{Compactness for perturbed PDEs}\label{sec_holder}

In this section we produce preliminary levels of compactness for the solutions to a variant of \eqref{sigma-eq}, by proving that bounded solutions to a family of equations, that are parametrized by vectors of $\mathbb{R}^d$, are uniformly locally H\"older-continuous. 

This will be attained by means of classical viscosity methods, developed in \cite[Section VII]{Ishii-Lions}, and we will carry all the details over as a courtesy to the readers. We recall a Lemma due to Ishii and Lions: 

\begin{lemma}\label{IshiiLions}
Let $G \colon B_1\times\mathbb{R}^d\times\mathcal{S}(d)\to\mathbb{R}$ be a degenerate elliptic operator. Let $\Omega \subset B_1$, $u \in  C(B_1)$ and $\psi$ be twice continuously differentiable in a neighborhood of $\Omega\times \Omega$. Set 
\[
w(x, y): = u(x) - u(y) \quad \text{for} \quad (x, y) \in \Omega\times \Omega .
\]
If the function $w - \psi$ attains the maximum at $(x_0, y_0) \in \Omega\times \Omega$, then for each $ \varepsilon > 0$, there exist $ X, Y \in {\mathcal S}(d)$ such that
\begin{equation*}\label{viscosity_solution}
G(x_0, D_x \psi(x_0, y_0), X) \leq 0 \leq G(y_0, D_y \psi(x_0, y_0), Y),
\end{equation*}
and the block diagonal matrix with entries X, Y satisfies
\begin{equation*} \label{matrix_inequality}
- \left( \frac{1}{\varepsilon} + \|A \|\right) I \leq 
\left( 
\begin{array}{cc}
X & 0 \\
0 & -Y \\
\end{array}
\right)
\leq  A + \varepsilon A^2, 
\end{equation*}
where $A = D^2 \psi(x_0, y_0)$.
\end{lemma}

For a proof of this Lemma, we refer the reader to  \cite[Proposition II.3]{Ishii-Lions}, see also \cite[Theorem 3.2]{Crandall-Ishii-Lions1992}. We are ready to state and prove the H\"older regularity for solutions of the perturbed equation.

\begin{proposition}[H\"older-continuity]\label{teo_lipschitz} 
Let $u \in {C}(B_1)$ be a normalized viscosity solution to 
\begin{equation}\label{eq_pdewxi} 
	\mathscr{F}(Du + \xi, D^2u) \,=\,f \hspace{.4in} \mbox{in} \hspace{.1in} B_1
\end{equation}
where $\xi\in\mathbb{R}^d$ is \emph{arbitrary} and $\mathscr{F}(\vec{p},M) = \sigma(|\vec{p}|)F(M)$. Suppose A\ref{assump_F}, \eqref{normalization-sigma}, and A\ref{assump_f} are in force. Then, $u$ is locally H\"older-continuous in $B_1$. In addition, there exists $C>0$, depending on dimension, ellipticity constants and $\|f\|_\infty$, but not on $\xi\in\mathbb{R}^d$, such that
\[
	 \sup_{\substack{x, y \;\in \;B_{1/2} \\  x \not= y} }\frac{|u(x)\,- \,u(y)|}{|x\,-\, y|^\gamma}\, \leq \,C,
\]
for some $\gamma \in(0,1)$, universal, i.e. depending only on dimension, $\lambda$ and $\Lambda$. 
\end{proposition}
\begin{proof} As commented, the proof follows standard methods in viscosity theory. We will carry all details for completeness. 

Fix a $0< \chi < 1$ and define $\omega \colon \mathbb{R}^+\to\mathbb{R}^+$ be defined as $\omega(t):=t^\chi$. For some $0<r\ll 1$ to be determined further in the proof, we consider the quantity
\[
	\mathcal{L}\,:= \,\sup_{x, y \in B_r}(u(x) \,- \,u(y)-L_1 \omega(|x-y|) -L_2(|x\, -\,x_0|^2\,+\, |y\,-\,x_0|^2)),
\]
defined for every $x_0 \in B_{r/2}$. If we prove that  $\mathcal{L}\leq0$ for some appropriate choices of $L_1,\,L_2>0$, we establish the result. As it is usual when resorting to this class of arguments, we reason through a contradiction argument. That is to say the following:  suppose for every $L_1>0$ and $L_2>0$, there is $x_0\in B_{r/2}$ for which $\mathcal{L}>0$. In what follows, we split the proof in several steps.

\bigskip

\noindent{\bf Step 1 - }Consider $\psi,\, \phi \colon \overline{B_r} \times \overline{B_r} \rightarrow \mathbb{R}$, defined by
\[
\psi(x,y):= L_1\omega(|x-y|) + L_2(|x -x_0|^2+ |y -x_0|^2)
\]
and 
\[
\phi(x,y):=u(x) - u(y) - \psi(x,y).
\]
Let $(\bar{x}, \bar{y}) \in \overline{B_r} \times \overline{B_r}$ be a maximum point for $\phi$. Thus
\[
	\phi(\bar{x}, \bar{y}) \,=\, \mathcal{L}\,>\,0.
\]
We therefore conclude
\[
	\psi(\bar{x}, \bar{y})\,<\, u(\bar{x}) \,-\, u(\bar{y})\, \leq\, \osc_{B_1} u\,\leq \,2.
\]
It follows that
\[
	L_1\omega(|\bar{x} \,- \,\bar{y}|) \,+\, L_2(|\bar{x} \,-\, x_0|^2 \,+\, |\bar{y}\, -\, x_0|^2) \,\leq \,2.
\]
As usual, at this point we choose $L_2$ as to ensure that $\bar{x}$ and $\bar{y}$ are interior points. In fact, if
\[
	L_2\,:=\, \left(\frac{4\sqrt{2}}{r}\right)^2
\]
we get
\[
	|\bar{x}\,- \,x_0| \,\leq \,\frac{r}{4} \hspace{.2in} \mbox{and} \hspace{.2in} |\bar{y}\, -\, x_0| \,\leq\, \frac{r}{4},
\]
hence concluding $\bar{x},\,\bar{y}\in B_r$. Finally, it is straightforward to notice that $ \bar{x} \neq \bar{y}$; otherwise, we would have $\mathcal{L}\leq 0$ trivially.

\bigskip

\noindent{\bf Step 2 - } At this point, we resort to the  the Ishii-Lions Lemma, stated in Lemma \ref{IshiiLions}. We proceed by computing $D_x\psi$ and $D_y\psi$ at $(\bar{x}, \bar{y})$. We find
\[
 D_x \psi( \bar{x},  \bar{y})= L_1 {\omega}^{\prime}(| \bar{x}-  \bar{y}|)| \bar{x} -  \bar{y}|^{-1}( \bar{x} -  \bar{y}) + 2 L_2( \bar{x} - x_0),
\]
and
\[
- D_y \psi( \bar{x},  \bar{y})= L_1 {\omega}^{\prime}(| \bar{x}-  \bar{y}|)| \bar{x} -  \bar{y}|^{-1}( \bar{x} -  \bar{y}) - 2 L_2( \bar{x} - x_0).
\]
For ease of presentation, we introduce the following notation:
\[
{\xi}_{ \bar{x}}:= D_x \psi( \bar{x},  \bar{y})\quad \text{and}\quad
{\xi}_{ \bar{y}}:= D_y \psi( \bar{x},  \bar{y}).
\]

From Lemma \ref{IshiiLions} we learn that for every $\varepsilon>0$, there are matrices $X,\,Y\in\mathcal{S}(d)$ satisfying the viscosity inequalities
\begin{equation}\label{eq1_IL}
\sigma(|{\xi}_{ \bar{x}} + \xi|) F(X) - f(\bar{x})\leq 0 \leq \sigma(|{\xi}_{ \bar{y}} + \xi|) F(Y) - f(\bar{y}).
\end{equation}
In addition, 
\begin{equation} \label{eq2_IL}
\left( 
\begin{array}{cc}
X & 0 \\
0 & -Y \\
\end{array}
\right)\,
\leq \,
\left( 
\begin{array}{cc}
Z & -Z \\
-Z & Z \\
\end{array}
\right)\,
+\,
2L_2I \,+ \,\varepsilon A^2,
\end{equation}
where $A:= D^2 \psi (\bar{x}, \bar{y})$ and
\[
Z:= L_1 {\omega}^{\prime \prime}(|\bar{x} - \bar{y}|) \dfrac{(\bar{x} - \bar{y})\otimes(\bar{x} - \bar{y}) }{|\bar{x} - \bar{y}|^2} + L_1\dfrac{{\omega}^{\prime}(|\bar{x} - \bar{y}|)}{|\bar{x} - \bar{y}|}\left( I - \dfrac{(\bar{x} - \bar{y})\otimes(\bar{x} - \bar{y}) }{|\bar{x} - \bar{y}|^2}\right).
\]

\bigskip

\noindent{\bf Step 3 - }Next we apply the matrix inequality \eqref{eq2_IL} to suitable vectors to recover information on the eigenvalues of $X\,-\,Y$. Let $v \in \mathbb{S}^{d-1}$ and consider first $(v, v)\in\mathbb{R}^{2d}$; we obtain
\[
\langle(X-Y)v, v \rangle \leq (4L_2 + 2\varepsilon \eta),
\]
where $\eta := \|A^2\|$. It is consequential that all eigenvalues of $X-Y$ are bellow $4L_2 + 2\varepsilon \eta$. Furthermore, we apply \eqref{eq2_IL} to vectors of the form $(z, -z)\in\mathbb{R}^{2d}$, where
\[
	z\,:=\,\frac{\bar{x}\,-\,\bar{y}}{|\bar{x}\,-\,\bar{y}|};
\]
we then get
\begin{equation}\label{eq3_IL}
\langle(X-Y)z, z \rangle \leq 4 L_1  {\omega}^{\prime \prime}(|\bar{x} - \bar{y}|) + (4L_2 + 2\varepsilon \eta)|z|^2.
\end{equation}
From the definition of $\omega$, we learn it is twice differentiable, $\omega >0$ and ${\omega}^{\prime \prime} <0$. It then follows from \eqref{eq3_IL} that at least one eigenvalue of $X-Y$ is bellow $-4L_1 + 4L_2 + 2\varepsilon \eta$. Observe that this quantity will be negative for $L_1$ sufficiently large. In the sequel, we  compute 
\[
	\mathcal{M}_{\lambda, \Lambda}^{-}(X\,-\, Y)\, \geq \,4 \lambda L_1 \,-\, (\lambda \,+\, (d\,-\,1)\Lambda)(4L_2 \,+ \, 2\varepsilon \eta);
\]
this inequality builds upon the definition of ellipticity and \eqref{eq1_IL} to produce

\begin{equation}\label{eq7_IL}
	4\lambda L_1\, \leq\, (\lambda \,+\, (d\,-\,1)\Lambda)(4L_2\, +\,  2\varepsilon \eta)\, +\, \dfrac{f(\bar{x})}{\sigma(|{\xi}_{\bar{x}}\, +\,\xi|)}\, -\, \frac{f(\bar{y})}{\sigma(|{\xi}_{\bar{y}}\, +\, \xi|)}.
\end{equation}

\bigskip

\noindent{\bf Step 4 - }At this point we examine two different cases. We start by considering $|\xi|>C_0$, where $C_0>0$ is yet to be determined. Estimate the norm of $\xi_{\bar{x}}$ as follows:
\begin{equation}\label{pokemon}
	|\xi_{\bar{x}}|\,\leq\, L_1|w'(|\bar{x}-\bar{y}|)|\,+\,2L_2\,\leq\,cL_1,
\end{equation}
for some constant $c>0$, universal. We choose next $C_0:=50cL_1 > 1.03$, for $L_1$ to be fixed later. Since $|\xi_{\bar{x}}|<cL_1$ and $|\xi|>50cL_1$ we get
\[
	|\xi\,+\xi_{\bar{x}}|\,\geq\, C_0\,-\,\frac{C_0}{50}\,=\,\frac{49}{50}C_0;
\]
a similar reasoning yields
\[
	|\xi\,+\xi_{\bar{y}}|\,\geq\, C_0\,-\,\frac{C_0}{50}\,=\,\frac{49}{50}C_0.
\]
The former inequalities, combined with the fact that $\sigma$ is nondecreasing, yield
\begin{equation}\label{eq10_IL}
\dfrac{f(\bar{x})}{\sigma(|\xi_{\bar{x}} + \xi|)} \leq \displaystyle \dfrac{\|f\|_{L^{\infty}(B_1)}}{\sigma\left(\dfrac{49C_0}{50}\right)}\leq \|f\|_{L^{\infty}(B_1)}
\end{equation}
and
\begin{equation}\label{eq11_IL}
\dfrac{- f(\bar{y})}{\sigma(|\xi_{\bar{y}} + \xi|)} \leq \displaystyle \dfrac{\|f\|_{L^{\infty}(B_1)}}{\sigma\left(\dfrac{49C_0}{50}\right)}\leq \|f\|_{L^{\infty}(B_1)}.
\end{equation}

\bigskip

\noindent On their turn, inequalities \eqref{eq10_IL} and \eqref{eq11_IL} combined with \eqref{eq7_IL} produce

\begin{equation}\label{eq12_IL}
4\lambda L_1 \leq (\lambda  + (d -1) \Lambda)(4L_2 + 2\varepsilon\eta)  + 2 \|f\|_{L^{\infty}(B_1)}.
\end{equation}
By choosing $L_1=L_1(\lambda, \Lambda, d,L_2,r)\gg 1$ sufficiently large, we obtain a contradiction. Consequential on this contradiction is the fact that $\mathcal{L}\leq 0$; hence, we obtain local $\chi$-H\"older continuity of the solutions in the case $|\xi|>C_0$.

\bigskip

\noindent{\bf Step 5 - }Consider now the complementary case; i.e., let $\left|\xi\right|\,\leq\,C_0$, where $C_0  = 50  cL_1$ was chosen in the previous step.
Define the operator
\[
	G(x,p,M)\,:=\ \sigma(|\xi\,+\,p|) F(M)\,-\,f(x).
\]
It follows that $G(x,p,M)$ is uniformly elliptic provided $|p|> 5C_0$. By using previous regularity results (see, for instance, \cite{D} or \cite{IS2}), we derive H\"older-continuity of the solutions. Gathered with the former step, this fact completes the proof of the theorem.
\end{proof}



Once compactness for  solutions of the $\xi$-perturbed equation is available, we approximate solutions to \eqref{sigma-eq} and \eqref{eq_pdewxi} by solutions to $F=0$. This is our next goal; however before we advance, we need first to introduce a new concept, which is the content of next section.

\section{Non-collapsing moduli of continuity} \label{sct non-collapsing}

In this section we formalize the notion of a family of non-collapsing moduli
of continuity. Hereafter we collect all intervals of the form $(0, T]$:
$$
	\mathscr{I} := \left \{ (0, T] ~ \big | ~ 0< T < \infty \right \} \cup  \left \{ \mathbb{R}_0^+ \right \} .
$$

\begin{definition}[Non-collapsing]\label{def_ncoll}
A set $\Gamma$ of moduli of continuity defined over an
interval $I \in \mathscr{I}$ is said to be non-collapsing if for all sequences $\left(  f_{n}\right)
_{n\in\mathbb{N}}\subset\Gamma,$ and all sequences of scalars $\left(
a_{n}\right)  _{n\in\mathbb{N}}\subset I$, we have%
\[
	f_{n}(a_{n})\rightarrow0 \quad \text{ implies } \quad   a_{n}\rightarrow0.
\]
\end{definition}

The former definition plays an important role in the tangential analysis developed in the paper. In fact, when one tries to connect the prospective regularity theory for $\sigma(|Du|) F(D^2u)=f$ with the one available for $F(D^2h)=0$, we aim at profiting from a  sort of \emph{cancellation} effect, to be understood in the viscosity sense. This is only achievable, however, if one carefully modulates the rate in which $\sigma(t)$ approaches zero, as $t\to 0$. Put differently, we must ensure the degeneracy law is not, itself, degenerate. 

\begin{definition} 
We define the collapsing measure of a family of moduli of continuity $\Gamma$ defined over an
interval $I \in \mathscr{I}$ as%
\[
	\mu\left(  \Gamma\right)  :=\sup\left\{  s\in I ~ \Big  |~  \inf\limits_{\sigma
	\in\Gamma}\sigma(s)=0\right\}  .
\]
\end{definition}
For obvious reasons all finite sets of moduli of continuity are non-collapsing,
and the interesting environment are infinite sets; for this reason in this
section all families of moduli of continuity shall be not finite.

It is not difficult to observe that the measure defined above characterizes
non-collapsing sets as follows:

\begin{proposition} \label{equiv}
Let $\Gamma$ be a family of moduli of continuity defined over an interval
$I \in \mathscr{I}$. The following are equivalent:

\begin{enumerate}
\item $\Gamma$ is non-collapsing.

\item For all sequences $(f_{n})_{n\in\mathbb{N}}\subset\Gamma$ and
$a\in I \setminus \{0\}$, $\liminf\limits_{n\rightarrow\infty}f_{n}(a)>0$.

\item $\mu\left(  \Gamma\right)  =0$.
\end{enumerate}
\end{proposition}

\begin{proof}
It is immediate that $(2)$ and $(3)$ are equivalent.

\noindent $(1)\Rightarrow(2).$ Suppose, seeking a contradiction, there was a sequence $\left(  f_{n}\right)  $ and a
certain $a>0$ such that%
\[
	\liminf\limits_{n\rightarrow\infty}f_{n}(a)=0.
\]
Hence, there is a subsequence $\left(  f_{n_{k}}\right)  $ such that%
\[
	f_{n_{k}}(a)\rightarrow0,
\]
and, since $\Gamma$ is non-collapsing, we conclude that $a=0$, which is a contradiction.

\medskip

\noindent $(2)\Rightarrow(1)$ Let us suppose, for the sake of contradiction, that
$\left(  1\right)  $ is not valid. Thus, there would exist $(f_{n}%
)_{n\in\mathbb{N}}\subset\Gamma$ and $(a_{n})_{n\in\mathbb{N}}\subset I$,
with $f_{n}(a_{n})\rightarrow0$ and $a_{n}\nrightarrow0$. So, up to a
subsequence, there exists a certain $a_{0}>0$ such that
\[
a_{n}\geq a_{0}>0.
\]
Since all the functions $f_{n}$ are non-decreasing, we would have%
\[
f_{n}\left(  a_{n}\right)  \geq f_{n}\left(  a_{0}\right)  >0
\]
and, recalling that $f_{n}(a_{n})\rightarrow0,$ we would have $f_{n}\left(
a_{0}\right)  \rightarrow0,$ which contradicts $(2)$.
\end{proof}

Observe that $\mu$ behaves as a kind of \textquotedblleft measure of
collapse\textquotedblright: for non-collapsing sets $\Gamma$, we have
$\mu(\Gamma)=0$ and for collapsing sets $\Gamma$ we have $\mu(\Gamma) >0$.
The higher the value of $\mu(\Gamma)$ the more degenerate the family $\Gamma$, otherwise refereed as 
\textquotedblleft more collapsing\textquotedblright.

Notice that
\[
\mu\left(  \Gamma_{1}\cup\Gamma_{2}\right)  =\max\left\{  \mu\left(
\Gamma_{1}\right)  ,\mu\left(  \Gamma_{2}\right)  \right\}  .
\]
However, for infinitely many unions it is possible that $\mu\left(  \Gamma
_{n}\right)  =0$ for all $n\in\mathbb{N}$ and
\[
\mu\left(  \bigcup\limits_{n=1}^{\infty}\Gamma_{n}\right)  = \left | I \right |,
\]
where $\left | I \right |$ stands for the total length of the interval $I$. For instance, say on $(0,1]$, take  $\sigma_{j} := t^j$ and define%
\[
\Gamma_{k}= \left \{\sigma_{1}, \cdots,\sigma_{k} \right \}
\]
for all $k\geq1$, we have $\mu\left(  \Gamma_{n}\right)  =0$; however 
\[
\mu\left(  \bigcup\limits_{n=1}^{\infty}\Gamma_{n}\right)  =1.
\]
A plenty of examples of non-collapsing sets of moduli of continuity can be
generated by the next propositions:

\begin{proposition} \label{equicontinuity}
If $\Gamma$ is a family of moduli of continuity $\sigma \colon I  \subset \mathbb{R}_0^+ \to
\mathbb{R}_0^+ $ such that set $\Gamma^{-1}:=\left\{  \sigma^{-1} ~ \big | ~ \sigma\in\Gamma\right\}$
is equicontinuous, then $\Gamma$ is non-collapsing.
\end{proposition}

\begin{proof}
If $\Gamma^{-1}$ is equicontinuous, given $\varepsilon>0$, there is a
$\delta>0$ such that%
\[
\left\vert x-y\right\vert <\delta\Rightarrow\sup_{\sigma^{-1}\in\Gamma^{-1}%
}\left\vert \sigma^{-1}\left(  x\right)  -\sigma^{-1}\left(  y\right)
\right\vert <\varepsilon.
\]
for all $x,y\in I.$ Thus%
\[
\left\vert \sigma\left(  x\right)  -\sigma\left(  y\right)  \right\vert
<\delta\Rightarrow\sup_{\sigma^{-1}\in\Gamma^{-1}}\left\vert \sigma
^{-1}\left(  \sigma\left(  x\right)  \right)  -\sigma^{-1}\left(
\sigma\left(  y\right)  \right)  \right\vert <\varepsilon,
\]
for all $x,y\in I$ and all $\sigma\in\Gamma$. Letting $y \to 0,$%
\[
\sigma\left(  x\right)  <\delta\Rightarrow x<\varepsilon
\]
for all $x\in I$ and all $\sigma\in\Gamma,$ i.e.,%
\[
x > \varepsilon \implies \inf_{\sigma\in\Gamma}\sigma\left(  x\right)
\geq\delta.
\]
Hence, $\Gamma$ is non-collapsing.
\end{proof}

\begin{proposition}\label{Dini implies non-collapsing} Let $\Gamma$ be a family of moduli of continuity and assume
$$
	S:=\sup\limits_{\omega \in \Gamma} \int_0^\tau \frac{\omega^{-1}(t)}{t} {\bf d} t < \infty,
$$
for some $\tau>0$, then $\mu\left (\Gamma \right ) = 0$.
\end{proposition}

\begin{proof} From Proposition \ref{equiv}, it suffices to show  $\omega_n \in \Gamma, ~ \omega_n (a) \to 0 \implies a= 0.$
Hence, let us suppose, seeking a contradiction, there exist a sequence $\omega_n \in \Gamma$ and a positive $a>0$ such that
$b_n := \omega_n(a) \to 0.$ We estimate
$$
	\displaystyle S\ge \displaystyle \int_0^\tau \frac{\omega_n^{-1}(t)}{t} {\bf d} t \ge \displaystyle \int_{b_n}^\tau \frac{\omega_n^{-1}(t)}{t} {\bf d} t \ge a \displaystyle \int_{b_n}^\tau \frac{1}{t} {\bf d} t \longrightarrow  +\infty,
$$
as $n \to 0$. We reach a contradiction, and Proposition \ref{Dini implies non-collapsing} is proven.
\end{proof}

Another way of producing a family of non-collapsing  moduli of continuity is through a sort of  ``shoring-up" process. 

\begin{definition}[Shoring-up] A sequence of moduli of continuity $(\sigma_{n})_{n\in\mathbb{N}}$ is said to be shored-up if there exists a sequence of positive numbers $(\gamma_{n})_{n\in\mathbb{N}}$ such that $\gamma_{n}\rightarrow0$ satisfying
\[
\inf\limits_{n}\sigma_{n}(\gamma_{n})>0,
\]
for every $n\in\mathbb{N}$.
\end{definition}

Here is a simple proposition relating the notion of shored-up sequence and non-collapsing moduli of continuity.

\begin{proposition}\label{prop_shoringup}
If a sequence of moduli of continuity $(\sigma_{n})_{n\in\mathbb{N}}$ is
shored-up then $\Gamma := \cup_{n\in\mathbb{N}} \left \{ \sigma_{n}  \right \}$ is non-collapsing.
\end{proposition}

\begin{proof}
For all $s>0,$ let $n_{s}$ be an integer such that $\gamma_{n}<s$ for all
$n>n_{s}.$ Since all the functions $\sigma_{n}$ are non-decreasing, we have
$\sigma_{n}\left(  \gamma_{n}\right)  \leq\sigma_{n}\left(  s\right)  $ for
all $n>n_{s}.$ Thus%
\[
0<\inf\limits_{n>n_{s}}\sigma_{n}(\gamma_{n})\leq\inf\limits_{n>n_{s}}%
\sigma_{n}\left(  s\right)  .
\]
Since $\sigma_{1}\left(  s\right)  >0,\ldots,\sigma_{n_{s}}\left(  s\right)
>0,$ we conclude that%
\[
\inf\limits_{n}\sigma_{n}\left(  s\right)  >0
\]
and $(\sigma_{n})_{n\in\mathbb{N}}$ is non collapsing.
\end{proof}

\section{Tangential analysis}\label{sec_gta}

In this section we establish an approximation result, relating \eqref{sigma-eq} and \eqref{eq_pdewxi} with the solutions to the homogeneous, uniformly elliptic, problem $F=0$. The approximating function whose existence is ensured by the next proposition plays a pivotal role in producing oscillation controls for the solutions to \eqref{sigma-eq}.

In what follows, we translate A\ref{assump_f} into a smallness condition for the source term $f$. In fact, throughout this section, we require 
\begin{equation}\label{eq_smallness}
	\left\|u\right\|_{L^\infty(B_1)}\,\leq\,1 \hspace{.4in}\mbox{and}\hspace{.4in}\left\|f\right\|_{L^\infty(B_1)}\,<\,\varepsilon,
\end{equation}
for some $\varepsilon>0$ yet to be determined. To see the conditions in \eqref{eq_smallness} are not restrictive, consider the function
\[
	v(x)\,:=\,\frac{u(rx)}{K},
\]
for $0<r\ll 1$ and $K>0$ to be defined. Notice that $v$ satisfies
\begin{equation}\label{eq_bareq}
	\overline{\sigma}\left(|Dv|\right)\overline{F}(D^2v)\,=\,\overline{f}\hspace{.4in}\mbox{in}\hspace{.1in}B_1,
\end{equation}
where
\[
	\overline{\sigma}(t)\,:=\,\sigma\left(\frac{K}{r}t\right),\hspace{.4in}\overline{F}(M)\,:=\,\frac{r^2}{K}F\left(\frac{K}{r^2}M\right)
\]
and
\[
	\overline{f}(x)\,:=\,\frac{r^2}{K}f(rx).
\]
Notice that 
\[
	\overline{\sigma}^{-1}(t)\,:=\,\frac{r}{K}\sigma^{-1}(t).
\]
Indeed,
\[
	\overline{\sigma}^{-1}(\overline{\sigma}(t))\,=\,\overline{\sigma}^{-1}\left(\sigma\left(\frac{K}{r}t\right)\right)\,=\,\frac{r}{K}\sigma^{-1}\left(\sigma\left(\frac{K}{r}t\right)\right)\,=\,t.
\]
Hence, by choosing $r<K$, it follows easily that
\[
	\int_0^1 \frac{\overline{\sigma}^{-1}(t)}{t} {\bf d}t \le \int_0^1 \frac{{\sigma}^{-1}(t)}{t} {\bf d}t \quad \text{ and} \quad \overline{\sigma}(1)\,=\,\sigma\left(\frac{K}{r}\right)\,\geq\,\sigma(1)\,\geq\,1.
\]
Hence, $\overline{\sigma}$ meets Assumption \ref{assump_dini}. Clearly, $\overline{F}$ is a $(\lambda,\Lambda)$-elliptic operator. Finally, by setting
\[
	r\,:=\,\varepsilon\hspace{.4in}\mbox{and}\hspace{.4in}K\,:=\,\frac{1}{\left\|u\right\|_{L^\infty(B_1)}\,+\,\left\|f\right\|_{L^\infty(B_1)}}
\]
we produce \eqref{eq_smallness} and find that \eqref{eq_bareq} falls within the same class as \eqref{sigma-eq}. 
\begin{proposition} \label{approx_lemma}
Let $\mathfrak{S}$ be a  set of non-collapsing moduli of continuity satisfying \eqref{normalization-sigma} and $u \in {C}(B_1)$ be a normalized viscosity solution of an equation of the form
\[
	{\sigma} \left ( \left|  D u + \xi \right | \right )\cdot F\left (D^2 u \right )=f \quad \mbox{in} \quad B_1,
\]
 where $\xi \in \mathbb{R}^d$, $\sigma \in \mathfrak{S}$, $F$ satisfies A\ref{assump_F},  and  $f$ verifies A\ref{assump_f}. 
Given $\delta >0$, there exists $\varepsilon=\varepsilon(\delta, \lambda,\Lambda,\mathfrak{S})>0$ such that if $f \in B_\varepsilon(L^\infty(B_1))$ then we can find a function $h \in B_L\left ( {C}^{1,\beta}(B_{1/2}) \right )$ such that
$$
	d_{L^\infty(B_{1/2})} \left (u, h \right ) < \delta,
$$
where $L$ and $\beta$ are universal numbers, in particular independent of $\mathfrak{S}$, $\delta$ and $\epsilon$.
\end{proposition}
\begin{proof} For ease of presentation we split the proof in five steps.

\bigskip

\noindent{\bf Step 1 -} Suppose the thesis of the lemma fails to hold. Then there exist sequences $({\sigma}_j)_{j \in \mathbb{N}}, ({\xi}_j)_{j \in \mathbb{N}},(u_j)_{j \in \mathbb{N}}, (F_j)_{j \in \mathbb{N}}, (f_j)_{j \in \mathbb{N}}$ and a number ${\delta}_0 > 0$ such that, for every $j\in\mathbb{N}$, we have
\begin{enumerate}
\item $F_j \colon \mathcal{S}(d)\to\mathbb{R}$ is a $(\lambda,\Lambda)$-elliptic operator;
\item $\sigma_j$ is a modulus of continuity satisfying $\sigma_j(0)=0$ and $\sigma_j(1)\geq1$. In addition, if $\sigma_j(a_j)\to 0$ then $a_j\to 0$.  
\item $f_j\in L^\infty(B_1)$ is such that
\[
	\left\|f_j\right\|_{L^\infty(B_1)}\,<\,\frac{1}{j};
\]
\item 
\begin{equation}\label{eq_pedej}
	\sigma_j(|Du_j + \xi_j|)F_j(D^2 u_j)\,=\,f_j \hspace{.4in} \mbox{in} \hspace{.1in} B_1,
\end{equation}
\end{enumerate}
however, 
\begin{equation}\label{eq_contradictionapprox}
	\sup_{x \in B_{1/2}}\left|u_j(x)\,-\,h(x)\right|\,\geq\,\delta_0
\end{equation}
for every $h\in B_L \left ( {C}^{1,\beta}_{loc}(B_1) \right )$, with $\beta >0$ and $L>1$ to be chosen.
\bigskip

\noindent{\bf Step 2 -} Because ellipticity is uniform along the sequence $(F_j)_{j\in\mathbb{N}}$, it follows that $F_j\to F_\infty$ as $j\to\infty$, through a subsequence if necessary. In addition, it follows from Proposition \ref{teo_lipschitz} that $(u_j)_{j\in\mathbb{N}}$ converges uniformly to a function $u_\infty$. Our goal is to prove that 
\[
	F_\infty(D^2u_\infty)\,=\,0 \hspace{.4in} \mbox{in} \hspace{.1in} B_{9/10}.
\]
To that end, introduce the second order polynomial $p(x)$, defined as
\[
	p(x)\,:=\,u_\infty(y)\,+\,b\cdot (x\,-\,y)\,+\,\frac{1}{2}(x\,-\,y)^TM(x\,-\,y);
\]
it is clear that $p(y)=u_\infty(y)$; suppose without loss of generality that $p(x)\leq u_{\infty}(x)$ for $x\in B_{3/4}$. Our goal is to verify that 
\begin{equation}\label{eq_approxgoal}
	F_\infty(M)\leq 0. 
\end{equation}

\bigskip

\noindent{\bf Step 3 - } For $0<r\ll 1$ fixed, let $(x_j)_{j\in\mathbb{N}}$ be defined by
\[
	p(x_j)\,-\,u_j(x_j)\,:=\,\max_{x\in B_r}\left(p(x)\,-\,u_j(x_j)\right).
\]
We infer from \eqref{eq_pedej} that 
\[
	\sigma_j(\left|b\,+\,\xi_j\right|)F_j(M)\,\leq\,f_j(x_j).
\]
If $(\xi_j)_{j\in\mathbb{N}}$ is an unbounded sequence, consider the (renamed) subsequence satisfying $|\xi_j|>j$, for every $j\in\mathbb{N}$. There exists $j^*\in\mathbb{N}$ such that 
\[
	|b\,+\,\xi_j|\,>\,1
\]
for every $j>j^*$. From Assumption \ref{assump_dini} we have
\[
	F_j(M)\,\leq\,\sigma_j(|b\,+\,\xi_j|)F_j(M)\,\leq\,f_j(x_j),
\]
for $j>j^*$. By letting $j\to\infty$, we obtain \eqref{eq_approxgoal}. Conversely, if $(\xi_j)_{j\in\mathbb{N}}$ is bounded, at least through a subsequence
\[
	b\,+\,\xi_j\,\longrightarrow b\,+\,\xi^*.
\]
If $|b\,+\,\xi^*|>0$, we know $\sigma_j(|b+\xi_j|)\nrightarrow 0$. Hence
\[
	F_j(M)\,\leq\,\frac{f_j(x_j)}{\sigma_j(|b\,+\,\xi_j|)}\,\longrightarrow\,0
\]
and \eqref{eq_approxgoal} follows. If, on the other hand, $|b\,+\,\xi^*|=0$, we distinguish two cases. The first is $b\equiv 0$ and $\xi_j\to 0$. If there is a subsequence $(\xi_j)_{j\in\mathbb{N}}$ for which $\xi_j\neq 0$, the previous reasoning applies and the argument is complete. 

On the opposite, it can be $b=\xi_j=0$ for every $j\in\mathbb{N}$, sufficiently large. This case is tackled in the next step.

\bigskip

\noindent{\bf Step 4 - }We work under the assumption $b\equiv\xi_j\equiv 0$. Notice that if $\spec(M)\subset(-\infty,0]$, ellipticity produces \eqref{eq_approxgoal}; in fact
\[
	F_{\infty}(M)\,\leq\,\lambda\sum_{i=1}^d\tau_i\,\leq\,0,
\]
where $\left\lbrace\tau_i,\;i=1,\,\ldots,\,d\right\rbrace$ are the eigenvalues of $M$. Hence, we also suppose $M$ has $k>0$ strictly positive eigenvalues. Let $(e_i)_{i=1}^k$ be the associated eigenvectors and define
\[
	E\,:=\,\Span\left\lbrace e_1,\,e_2,\ldots,\,e_k \right\rbrace.
\]
Consider the orthogonal sum $\mathbb{R}^d=E\oplus G$ and the orthogonal projection $P_E$ on $E$. Define the test function 
\[
	\varphi(x)\,:=\,\kappa\sup_{e\in\mathbb{S}^{d-1}}\left\langle P_Ex,e\right\rangle\,+\,\frac{1}{2}x^TMx.
\]
Because $u_j\to u_\infty$ locally uniformly, and $2^{-1}x^TMx$ touches $u_\infty$ at zero, the stability of minimizers implies that $\varphi$ touches $u_j$ at $x_j^\kappa\in B_r$, for every $0<\kappa\ll 1$ and  $j\gg 1$. 

Suppose $x^\kappa_j\in G$. In this case, $\varphi$ touches $u_j$ at $x_j^\kappa$, regardless of the direction $e\in\mathbb{S}^{d-1}$. It follows that 
\[
	\sigma_j\left(|Mx^\kappa_j\,+\,\kappa e|\right)F_j(M)\,\leq\,f_j(x_j)
\]
for every $e\in\mathbb{S}^{d-1}$. By taking supremum with respect to the direction $e$ on both sides of the former inequality, and noticing that
\[
	\kappa\,\leq\,\sup_{e\in\mathbb{S}^{d-1}}|Mx^\kappa_j\,+\,\kappa e|,
\]
we obtain
\[
	F_j(M)\,\leq\,\frac{f_j(x^\kappa_j)}{\sigma_j(\kappa)}\,\longrightarrow \, 0
\]
as $j\to \infty$.  To complete the proof we focus on the case $P_Ex^\kappa_j\neq 0$. Here
\[
	\sup_{e\in\mathbb{S}^{d-1}}\left\langle P_Ex^\kappa_j,e\right\rangle\,=\,|P_ex^\kappa_j|.
\]
From the information available for $u_j$, we obtain
\[
	\sigma_j\left(\left|Mx^\kappa_j\,+\,\kappa\frac{P_Ex^\kappa_j}{|P_Ex^\kappa_j|}\right|\right)F_j\left(M\,+\,\kappa\left(Id\,+\,\frac{P_Ex^\kappa_j}{|P_Ex^\kappa_j|}\otimes\frac{P_Ex^\kappa_j}{|P_Ex^\kappa_j|}\right)\right)\,\leq\,f_j(x_j^\kappa).
\]

Write $x^\kappa_j$ as
\[
	x^\kappa_j,=\,\sum_{i=1}^da_ie_i,
\]
where $\left\lbrace e_i,\;i=1,\,\ldots,\,d\right\rbrace$ are the eigenvectors of $M$. Hence,
\[
	Mx^\kappa_j\,=\,\sum_{i=1}^k\tau_ia_ie_i\,+\,\sum_{i=k+1}^d\tau_ia_ie_i,
\]
with $\tau_i>0$ for $i=1,\,\ldots,\,k$. We then obtain
\[
	\begin{split}
		\kappa\,&\leq\,\kappa\,+\,\frac{1}{|P_Ex^k_j|}\sum_{i=1}^k\tau_ia_i^2\,\leq\, \kappa\,+\,\frac{1}{|P_Ex^k_j|}\left\langle \sum_{i=1}^d\tau_ia_ie_i,\sum_{i=1}^k\tau_ia_ie_i\right\rangle\\
			&\leq\,\left\langle Mx^\kappa_j\,+\,\kappa\frac{P_Ex^\kappa_j}{|P_Ex^\kappa_j|},\frac{P_Ex^\kappa_j}{|P_Ex^\kappa_j|}\right\rangle\\
			&\leq\left |Mx^\kappa_j\,+\,\kappa\frac{P_Ex^\kappa_j}{|P_Ex^\kappa_j|}\right|.
	\end{split}
\]
Once again we get
\[
	F_j(M)\,\leq\,F_j\left(M\,+\,\kappa\left(Id\,+\,\frac{P_Ex^\kappa_j}{|P_Ex^\kappa_j|}\otimes\frac{P_Ex^\kappa_j}{|P_Ex^\kappa_j|}\right)\right)\,\leq\,\frac{f_j(x^\kappa_j)}{\sigma_j(\kappa)}\,\longrightarrow \, 0
\]
as $j\to \infty$.

\bigskip

\noindent{\bf Step 5 - }Hence, we conclude that $F_\infty(M)\leq 0$ and, therefore, $u_\infty$ is a supersolution to $F_\infty=0$ in the viscosity sense. To verify that $u_\infty$ is also a subsolution is analogous and we omit the details. Standard results in the regularity theory of viscosity solutions to homogeneous elliptic equations, \cite{Trundiger} and \cite{Caffarelli}, yield $u_\infty\in{C}^{1,\beta}_{loc}(B_1)$ for some $\beta\in(0,1)$. By setting $h:=u_\infty$ we obtain a contradiction and complete the proof.
\end{proof}

\section{Existence of approximating hyperplanes} \label{sct Approx}

Let us move forward with the proof of Theorem \ref{main_theorem1}. Hereafter let $L>0$ and $0<\beta < 1$ be the universal numbers from Proposition \ref{approx_lemma}.

As to ease the presentation, let us define two new moduli of continuity:
$$
	\gamma(t):=t\sigma(t)\quad \text{ and} \quad \omega(t):=\gamma^{-1}(t).
$$
Next we make a first choice of constants $0<r < \mu_1 < 1$, by dividing the analysis in two cases:

\medskip

\noindent {\bf Case 1.} If $t^\beta=o(\omega (t))$, we choose $0< r <1/2$ so small that
$$
	2Lr^\beta = \omega(r) =: \mu_1 > r.
$$
This is the most interesting case, for which the degeneracy law is stronger than $t^{\frac{1}{\beta} - 1}$. 

\noindent {\bf Case 2.} If $\omega(t)=O(t^\beta)$, we fix $0< \alpha < \beta$ and make  $0< r <1/2$ so small that
$$
	2Lr^\beta =  r^\alpha  =: \mu_1 > r.
$$
Notice that, once fixed $0< \alpha < \beta$, the above choice becomes universal. 

In what follows we shall treat both cases concomitantly. Define, hereafter, the ratio 
$$
	0< \theta := \frac{r}{\mu_1} <1.
$$
Next, under Assumption A\ref{assump_dini}, we know the sequence 
$$
	\left ( a_k \right )_{k\in \mathbb{N}} := \left ( \sigma^{-1} \left (\theta^k \right ) \right )_{k\in \mathbb{N}}
$$
belongs to $\ell_1$. We apply Lemma \ref{DP} to the sequence $\left ( a_k \right )_{k}$ with, $0<\delta < \frac{1}{10}$ fixed and $0<\epsilon < 1$ chosen in such a way 
$$
	\epsilon \left (1+\delta \right) = 1
$$
This creates a sequence of positive numbers $\left (c_k \right )_{k} \in c_0$ for which 
\begin{equation}\label{summable}
	\frac{19}{22}  \sum\limits_{i=1}^\infty  \sigma^{-1} \left (\theta^k \right ) \le \sum\limits_{i=1}^\infty  \frac{\sigma^{-1} \left (\theta^k \right )}{c_k} \le  \sum\limits_{i=1}^\infty  \sigma^{-1} \left (\theta^k \right ).
\end{equation}
In the sequel, we generate a shored-up sequence of moduli of continuity by the following  recursive formula:
\begin{equation}\label{rec1}
	\begin{array}{lll}
		\displaystyle \sigma_0(t) &=& \displaystyle  \sigma(t); \vspace{0.05in} \\
		\displaystyle  \sigma_1(t) &=&\displaystyle   \frac{\mu_1}{r} \sigma(\mu_1 t); \vspace{0.05in} \\
		\displaystyle  \sigma_2(t) &=& \displaystyle  \frac{\mu_1\mu_2}{r^2} \sigma(\mu_1\mu_2 t); \vspace{0.05in} \\
		& \displaystyle  \vdots&	\vspace{0.05in} \\
		\displaystyle  \sigma_n(t) &=&\displaystyle   \frac{\mu_1\mu_2\cdots\mu_n}{r^n} \sigma(\mu_1\mu_2\cdots\mu_n t),
	\end{array}
\end{equation}
where $\mu_1>r$ has already been chosen and for  $k\ge 2$,  the value of $\mu_k$ is determined through the following new algorithm:

\medskip

\noindent {\it If}
$$
	 \frac{\mu_1^2}{r^2} \sigma \left ( \mu_1^2 \cdot c_2 \right ) \ge 1,
$$ 
\noindent {\it then}
$$
	\mu_2 = \mu_1;
$$
\noindent {\it otherwise} 
$$
	\mu_1 < \mu_2 < 1
$$
is chosen such that
$$
	 \frac{\mu_1\mu_2}{r^2} \sigma \left ( (\mu_1 \mu_2) \cdot c_2 \right ) = 1,
$$ 
where $c_2$ is the $2$nd element of the sequence $\left (c_k \right )_{k} \in c_0$ for which \eqref{summable} is verified. 

\medskip

Next we apply the above algorithm recursively, that is: once chosen $r < \mu_1 \le \mu_2 \le \cdots \le \mu_k$ we decide on the value of $\mu_{k+1}$ as:
$$
	\text{ if selecting }  \mu_{k+1} = \mu_k \text{ yields } \sigma_{k+1}(c_{k+1}) \ge 1,
$$
we set $\mu_{k+1} = \mu_k$. Otherwise, $\mu_{k+1} > \mu_k$ is chosen such that
$$
	\sigma_{k+1} \left (c_{k+1} \right ) = 1,
$$
where, as before, $c_{k+1}$ is the $(k+1)$th element of the sequence $\left (c_k \right )_{k} \in c_0$ crafted in Lemma \ref{DP}, for which \eqref{summable} holds. 

Let $\mathfrak{S}$ denote the family of moduli of continuity generated through the described algorithm: 
$$
	\mathfrak{S} := \left \{ \sigma_0(t), \sigma_1(t), \cdots, \sigma_n(t), \cdots \right \}.
$$
According to Proposition \ref{prop_shoringup}, this is a non-collapsing family of moduli of continuity.

\begin{proposition}\label{step_1}
Let $u \in {C}(B_1)$ be a normalized viscosity solution to \eqref{eq_pdewxi}. Suppose A\ref{assump_F}, A\ref{assump_dini}, and A\ref{assump_f}  are in force. There exists an $\epsilon > 0$ such that if $\|f\|_{L^\infty(B_1)} < \epsilon$, then, one can find  an affine function $\ell(x)$ and a universal constant $C>0$ such that 
\[
	\ell(x)\,=\, a + b \cdot x,  \hspace{.3in} \mbox{with} \hspace{.3in} |a|\,+\,|b|\,\leq\,C
\]
and
\[
	\sup_{x \in B_r} \left| u(x) \,- \,\ell(x) \right|\, \leq \,  \mu_1 \cdot r.
\]
\end{proposition}

\begin{proof}
From Proposition \ref{approx_lemma} we infer the existence of $h\in{C}^{1,\beta}_{loc}(B_1)$ such that 
\[
	\sup_{x\in B_{9/10}}\left|u(x)\,-\,h(x)\right|\,\leq\,\delta,
\]
for some $\delta>0$, to be set further in the proof. As mentioned before, the regularity of the approximating function $h$ yields
\[
	\sup_{x\in B_r}\left|h(x)\,-\,h(0)\,-\,Dh(0)\cdot x\right|\,\leq\,Lr^{1\,+\beta}
\]
for a universal constant $L>0$ and every $0<r < 1/2$. By choosing $a:= h(0)$ and $b:=Dh(0)$ it is clear that both coefficients are universally bounded. A straightforward application of the triangular inequality yields
\[
	\begin{array}{lll}
		\sup_{x\in B_\rho}\left|u(x)\,-\,a\,-\,b\cdot x\right|\,&\leq& \,\delta\,+\,Lr^{1\,+\,\beta}  \\
		&= & \,\delta\,+ \dfrac{1}{2} \mu_1 \cdot r.
	\end{array}
\]
Choosing $\delta = \dfrac{1}{2} \mu_1 \cdot r$ sets the value of $\epsilon>0$, through Proposition \ref{approx_lemma}, and the proof is completed.
\end{proof}

The next proposition extends the statement in Proposition \ref{step_1} to arbitrarily small radii, in a discrete scale generated by a geometric sequence out from the original radius $0<r\ll 1$. Moving across those discrete scales involves a scaling argument. At this precise point of the argument, scaled solutions fail to satisfy the original equation \eqref{sigma-eq}. In turn, they satisfy
\[
	\mathscr{F}_n(D u_n +\xi_n, D^2 u_n)\,=\,f_n(x)\hspace{.4in}\mbox{in}\hspace{.1in}B_1,
\]
where, $\xi_n\in\mathbb{R}^d$ is arbitrary and at each scale the new operator $\mathscr{F}_n(\vec{p}, M)$ has law of degeneracy $\sigma_n$ and diffusion agent $F_n$. The switch from \eqref{sigma-eq} to \eqref{eq_pdewxi} is justified by the necessity of producing uniform compactness estimates available at this instance of the argument.

\begin{proposition}[Oscillation control at discrete scales]\label{prop_oscdiscrete}
Let $u \in {C}(B_1)$ be a normalized viscosity solution to \eqref{eq_pdewxi}. Suppose A\ref{assump_F}, A\ref{assump_dini}, and A\ref{assump_f} are in force. Then there exists a sequence of affine functions $(\ell_n)_{n\in\mathbb{N}}$ of the form
\[
	\ell_n(x)\,:=\,A_n\,+\,B_n\cdot x
\]
satisfying
\begin{equation}\label{eq_induction1}
	\sup_{x\in B_{r^n}}|u(x)\,-\,\ell_n(x)|\,\leq\,\left(\prod_{i=1}^n\mu_i\right)r^n,
\end{equation}	
\begin{equation}\label{eq_induction2}
	\left|A_{n+1}\,-\,A_n\right|\,\leq\,C\left(\prod_{i=1}^{n}\mu_i\right)r^{n}
\end{equation}	
and
\begin{equation}\label{eq_induction3}
	\left|B_{n+1}\,-\,B_n\right|\,\leq\,C\prod_{i=1}^{n}\mu_i
\end{equation}	
for every $n\in\mathbb{N}$.
\end{proposition}
\begin{proof}
We prove the proposition through an induction argument. As before, we proceed in steps.

\bigskip

\noindent {\bf Step 1 - }For $\mu_1$ and $\ell = \ell_0$ as in Proposition \ref{step_1},  consider the auxiliary function
\[
	u_1(x)\,:=\,\frac{u(rx)\,-\,\ell_0(rx)}{\mu_1r}.
\]
Notice that $u_1$ solves
\[
	\sigma_1\left(\left|Du_1\,+\,\frac{1}{\mu_1}D\ell\right|\right)F_1(D^2u_1)\,=\,f_1(x)\hspace{.4in}\mbox{in}\hspace{.1in}B_1,
\]
where
\[
	\sigma_1(t)\,:=\,\frac{\mu_1}{r}\sigma(\mu_1t),
\]
\[
	F_1(M)\,:=\,\frac{r}{\mu_1}F\left(\frac{\mu_1}{r}M\right)\hspace{.4in} \mbox{and} \hspace{.4in}f_1(x)\,:=\,f(rx).
\]
The selection through the algorithm preceding  Proposition \ref{step_1} ensures that $\sigma_1(1)=1$. Therefore, $u_1$ falls within the scope of this result and we infer the existence of an affine function $\ell_1$, with universal bounds, such that
\[
	\sup_{x \in B_r} \left| u_1(x) \,- \,\ell_1(x) \right|\, \leq \,r \mu_1.
\]
At this point, we define $u_2$ as
\[
	u_2(x)\,:=\,\frac{u_1(rx)\,-\,\ell_1(rx)}{\mu_2 r},
\]
for $r<\mu_1\le \mu_2$ chosen earlier. It is clear that $u_2$ satisfies 
\[
	\sigma_2\left(\left|Du_2\,+\,\frac{1}{\mu_1}D\ell_1\right|\right)F_2(D^2u_2)\,=\,f_2(x)\hspace{.4in}\mbox{in}\hspace{.1in}B_1,
\]
where, as before,
\[
	\sigma_2(t)\, =\,\frac{\mu_1\mu_2}{r^2}\sigma(\mu_1\mu_2t).
\]
The governing diffusion agent for $u_2$ is given by
\[
	F_2(M)\,:=\,\frac{r^2}{\mu_1\mu_2}F\left(\frac{\mu_1\mu_2}{r^2}M\right)
\]
and the source term $f_2(x)\,:=\,f(r^2x).$ Hence, $u_2$ meets the requirements of Proposition \ref{step_1}, which ensures the existence of an affine function $\ell_2$, with universal bounds, such that
\[
	\sup_{x \in B_r} \left| u_2(x) \,- \,\ell_2(x) \right|\, \leq \,r \mu_1.
\]

Proceeding inductively we notice that
\[
	u_{k+1}(x)\,:=\,\frac{u_k(rx)\,-\,\ell_k(rx)}{\mu_{k+1}r}
\]
solves an equation with degeneracy $\sigma_{k+1}$, given by
\[
	\sigma_{k+1}(t)\,=\,\frac{\mu_{k+1}}{r}\sigma_k\left(\mu_{k+1}t\right)\,=\,\frac{\prod_{i=1}^{k+1}\mu_i}{r^{k+1}}\sigma\left(\prod_{i=1}^{k+1}\mu_it\right).
\]
Recall, $\mu_{k+1}\ge \mu_k$ is determined in such way that either $\mu_{k+1} = \mu_k$ or else
\begin{equation}\label{key eq k+1}
	\sigma_{k+1}\left( c_{k+1}\right)\,=\,1.
\end{equation}
As before, we resort to Proposition \ref{step_1} to ensure the existence of an affine function $\ell_{k+1}$ satisfying
\[
	\sup_{x\in B_r}\left|u_{k+1}(x)\,-\,\ell_{k+1}(x)\right|\,\leq\,\mu_1 \cdot r.
\]

\bigskip

\noindent{\bf Step 2 - }Reverting back to the original solution $u$, we find
\[
	\sup_{x\in B_{r^{k}}}\left|u(x)\,-\,\ell_k(x)\right|\,\leq\,\left(\prod_{i=1}^{k}\mu_i\right) \cdot r^{k},
\]
where
\[	
	\begin{split}
		\ell_k(x)\,&:=\,\ell_1(x)\,+\,\sum_{i=2}^{k}\ell_i(r^{-1}x)\left(\prod_{i=1}^{k}\mu_i\right)r^i\\
			&=\,A_k\,+\,B_k\cdot x.
	\end{split}
\]
In addition, we have
\[
	\left|A_{k+1}\,-\,A_k\right|\,\leq\,C\left(\prod_{i=1}^{k-1}\mu_i\right)\cdot r^{k-1}
\]
and
\[
	\left|B_{k+1}\,-\,B_k\right|\,\leq\,C\left(\prod_{i=1}^{k-1}\mu_i\right),
\]
which completes the proof.
\end{proof}

\section{Convergence analysis} \label{sct conv. anal.}

In this final section we discuss the convergence of the approximating hyperplanes obtained in Section \ref{sct Approx}.  To ensure this fact, we must examine the summability of the series associated with $(A_n)_{n\in\mathbb{N}}$ and $(B_n)_{n\in\mathbb{N}}$. Such a convergence shall imply a modulus of continuity that takes the form of a sum, associated with the products $\Pi_{i=1}^n\mu_i$, which ultimately yields a proof of Theorem \ref{main_theorem1}.

\bigskip

\begin{proof}[Proof of Theorem \ref{main_theorem1}] 

The algorithm employed to craft the sequence $(\mu_n)_{n\in\mathbb{N}}$ is  key in the proof. There are two possibilities:

\medskip

\noindent Either the sequence stabilizes for some $k_0\ge 2$, that is 
\[
	\mu_{k_0} = \mu_{k_0+1} =  \mu_{k_0+2} = \cdots 
\]
or else for infinitely many $k$'s, there holds $\mu_{k} < \mu_{k+1}$. And when this happens:
\begin{equation}\label{key final proof}
	\frac{\prod_{i=1}^{k+1}\mu_i}{r^{k+1}}\sigma\left(\left[\prod_{i=1}^{k+1}\mu_i\right] c_{k+1} \right)\,=\,1.
\end{equation}

The former case falls into a classical setting, for which the convergence analysis yields in fact local ${C}^{1,\tau}$--regularity of solutions, for some $0<\tau < \beta$.  

Let us now investigate the latter case. Readily from \eqref{key final proof} one obtains
$$
	 {\sigma}_{k+1}\left( c_{k+1} \right)\,=\,1 \quad \iff \quad \frac{\prod_{i=1}^{k+1}{\mu}_i}{r^{k+1}}\sigma\left(\prod_{i=1}^{k+1} {\mu}_i \cdot c_{k+1}\right)
 = 1,
$$
which yields  
\begin{equation}\label{estimate-from Shoreup}
	\begin{array}{lll}
		\displaystyle \prod_{i=1}^{k+1} {\mu}_i  &=& \displaystyle \frac{1}{c_{k+1}} \sigma^{-1} \left ( \frac{r^{k+1}}{ \prod_{i=1}^{k+1}{\mu}_i} \right ) \vspace{0.05in} \\
		& \le & \displaystyle \frac{ \sigma^{-1} \left ( \theta^{k+1} \right )}{c_{k+1}}.
	\end{array}
\end{equation}
Estimate \eqref{summable} combined with estimate \eqref{estimate-from Shoreup} shows  the sequence 
$$
	\left ( \tau_{k} \right )_{k\in \mathbb{N}} :=\left (  \prod_{i=1}^{k} {\mu}_i \right )_{k\in \mathbb{N}}
$$ 
is summable and its $\ell_1$ norm is bounded by $\sum\limits_{i=1}^\infty \sigma^{-1}(\theta^i)$.

Therefore, it follows from \eqref{eq_induction2} and \eqref{eq_induction3} that $(A_n)_{n\in\mathbb{N}}$ and $(B_n)_{n\in\mathbb{N}}$ are Cauchy sequences. That is, there exist a real number $A_\infty$ and a vector $B_\infty$ such that
\[
	A_n\longrightarrow A_\infty\hspace{.5in}\mbox{and}\hspace{.5in}B_n\longrightarrow B_\infty.
\]
Set $\ell_\infty(x):=A_\infty+B_\infty\cdot x$. Observe also 
\[
	\left|A_\infty\,-\,A_n\right|\,\leq\,C\sum_{i=n}^\infty\tau_i r^n\hspace{.4in}\mbox{and}\hspace{.4in}\left|B_\infty\,-\,B_n\right|\,\leq\,C\sum_{i=n}^\infty\tau_i.
\]
For any $0<\rho\ll 1$ let  $n\in\mathbb{N}$ be such that 
$$
	r^{n+1}<\rho\leq r^n. 
$$
We then estimate
\[
	\begin{split}
		\sup_{x\in B_\rho}|u(x)\,-\,\ell_\infty(x)|\,&\leq\,\sup_{x\in B_{r^n}}|u(x)\,-\,\ell_n(x)|\,+\,\sup_{x\in B_{r^n}}|\ell_n(x)\,-\,\ell_\infty(x)|\\
			&\leq\,C\tau_nr^n\,+\,C\left(\sum_{i=n}^\infty\tau_i\right)r^n\\
			&\leq\, \frac{1}{r}C\left[\tau_n\,+\,\sum_{i=n}^\infty\tau_i\right]\rho\\
			&\leq\, \left(C\sum_{i=n}^\infty\tau_i\right)\rho.
	\end{split}
\]
Finally, set
\[
	\gamma(t)\,:=\,C\sum_{i=\lfloor \ln t^{-1}\rfloor}^\infty\tau_i,
\]
where $\lfloor M \rfloor := $ the biggest integer that is less than or equal to $M$. Since $\tau_i \in \ell_1$, $\gamma(t)$ is indeed a modulus of continuity. We have
\[
	\sup_{x\in B_\rho}|u(x)\,-\,u(0)\,-\,Du(0)\,\cdot\,x|\,\leq\,\gamma(\rho)\rho,
\]
and the proof of Theorem \ref{main_theorem1} is finally complete.
\end{proof}

\section{Final remarks}

 We start my commenting on the structural condition on $\sigma$. Throughout the paper we have assumed the law of degeneracy $\sigma$ is modulus of continuity. This, in particular, requires  $\sigma$ to be an increasing function, which might be a drawback for applications. Nonetheless, Theorem \ref{main_theorem1} remains true under the following relaxed condition: 
\[
	C^{-1}\rho(t)\leq\sigma(t)\leq C\rho(t),
\]
for some modulus of continuity $\rho$ whose inverse $\rho^{-1}$ is Dini continuous. Indeed, under such an assumption, if $u$ satisfies $\sigma(|Du|)F(D^2u) = f(x)$ in the viscosity sense, then 
$ \rho(|Du|)F(D^2u) = \frac{\sigma(|Du|)}{ \rho(|Du|)}f(x) := g(x) \in L^\infty$. 

Next we comment on the universality of the estimate provided in Theorem \ref{main_theorem1}. Let 
$$
	\Xi := \left \{ \sigma \colon I \to \mathbb{R}_0^{+} ~ \big | ~ \text{ is a modulus of continuity and } \sigma^{-1} \in L^1\left ((0, \tau]; \lambda^{-1} {\bf d} \lambda \right ) \right \}.
$$
For $\sigma \in \Xi$, let's denote:
$$
	 \left \| \sigma \right \|_\Xi :=  \sigma (1)  + \int_0^\tau \frac{\sigma^{-1}(\lambda)}{\lambda} {\bf d} \lambda.
$$
 Theorem \ref{main_theorem1} provides the existence of a modulus of continuity $\omega$ such that for any viscosity solution of
$\sigma (|Du|) F(D^2u) = f(x),  \text{ in } B_1,$ where $F$ is $(\lambda, \Lambda)$ uniform elliptic and $\|f\|_\infty \le C$, there holds:
$$
	\left | Du(x) - Du(y)\right | \le \omega \left ( |x-y|\right ),
$$
for all $x, y \in B_{1/4}$.  The dependence of $\omega$ with respect to $\sigma$ is, nonetheless, rather intricate, and in principle it does not depend solely upon the value $\|\sigma \|_{\Xi}$.  

It is  possible however to show  with the aid of Lemma \ref{DP-Comp}  that given a compact set $K \subset \Xi$ (with respect to $\|\cdot \|_{\Xi}$), there exists a universal modulus of continuity $\omega$, which depends only on $K$, $0< \lambda \le \Lambda$, and $C>0$, such that if $u$ is viscosity solution of
$$
	\sigma (Du) F(D^2u) = f(x), \quad \text{ in } B_1
$$
where $F$ is $(\lambda, \Lambda)$ uniform elliptic, $\|f\|_\infty \le C$ and $\sigma \in K$, then:
$$
	\left | Du(x) - Du(y)\right | \le \omega \left ( |x-y|\right ), 
$$
for all $x, y \in B_{1/4}$.

\bigskip

\noindent {\bf Acknowledgements:} PA was partially supported by CAPES and FAPERJ. 
DP was partially supported by CNPq 
and by Para\'iba State Research Foundation. 
EP was partly supported by CNPq, by FAPERJ  and by the Instituto Serrapilheira 
ET thanks support by UCF start-up grant. 

\bigskip

\noindent\textsc{P. Andrade},
Instituto Superior T\'ecnico, Universidade de Lisboa -- ULisboa, 1049-001, Av. Rovisco Pais, Lisboa, \\
\noindent Portugal
\hfill \texttt{pedra.andrade@tecnico.ulisboa.pt}

\bigskip

\noindent\textsc{D. Pellegrino},
Federal University of  Paraiba, Jo\~ao Pessoa-PB, \\
\noindent Brazil  
\hfill \texttt{daniel.pellegrino@academico.ufpb.br}

\bigskip

\noindent\textsc{E. Pimentel},
University of Coimbra,  CMUC Department of Mathematics, Coimbra, Portugal and Pontifical Catholic University of Rio de Janeiro -- PUC-Rio, Rio de Janeiro-RJ, Brazil\hfill \texttt{edgard.pimentel@mat.uc.pt}

\bigskip

\noindent\textsc{E. Teixeira}, University of Central Florida,  Orlando FL, \\
\noindent U.S.A.
\hfill \texttt{eduardo.teixeira@ucf.edu}

\end{document}